\documentclass{svjour3}             
\smartqed  
\usepackage{amsmath,amsbsy,amsfonts,amssymb}
\usepackage{latexsym,enumerate,verbatim,xspace,exscale}
\usepackage{graphicx}
\usepackage{graphics}

\usepackage{subfigure}
\usepackage{listings}
\usepackage{xcolor}

\begin{document}

\title{New links between PDE's and  Voronoi patterns}


\author{Yuriria Cortés-Poza         \and
        David Padilla-Garza         \and
        Pablo Padilla-Longoria
}

\institute{Y. Cortes-Poza \at
              IIMAS, Unidad Académica de Yucatán, Universidad Nacional Autónoma de México (UNAM), Yuc., México \\
              \email{yuriria.cortes@iimas.unam.mx}           
           \and
           D. Padilla-Garza \at
              Institute for Scientific Computing, Mathematics Department, TU Dresden, Dresden, Germany\\
              \email{david.padilla-garza@tu-dresden.de}
            \and
            P. Padilla-Longoria \at
                IIMAS, Mathematics and Mechanics Department,     Universidad Nacional Autonoma de México (UNAM), Ciudad Universitaria, México\\
                \email{pablo@mym.iimas.unam.mx}
              }

\date{Received: date / Accepted: date}

\maketitle

\begin{abstract}

This paper presents a range of results in partial differential equations (PDEs) in which Voronoi patterns arise.
We investigate the connection between the solution to an elliptic equation and its probabilistic interpretation as a stochastic colonization game. An agent-based model is designed and implemented to generate the Voronoi cells simulating experimental results with bacteria. We also consider the analytical solution to the problem, which enables us to define what we call a harmonic Voronoi tessellation.\\ 
We analyze parabolic equations in Riemannian manifolds, which have important applications in chemical reactions and diffusive fronts. By utilizing short-time heat kernel estimates, we demonstrate that the interaction of $n$ point sources gives rise to a Voronoi tessellation.\\
We recall some well-known results of wavefronts interactions from point light sources and the Huygens principle. We apply results about the particular set of weak solutions to the eikonal equation to characterize Voronoi patterns arising in this context as rectifiable sets.\\
Finally, we present an optimal transport problem and the corresponding Monge-Ampère equation, in which a uniform measure is transported to a sum of $n$ Dirac masses with a cost given by the Euclidean distance. These problems are naturally linked to power sets, a generalization of Voronoi tessellations. 


\end{abstract}

\section{Introduction}
Voronoi tessellations appear in many contexts and have critical applications in several fields, including archaeology, biology, computer science, economics, engineering, epidemiology, medicine, physics, image processing, and urban planning.

They go back as far as Descartes in 1644. He informally drew a diagram of what is now called a Voronoi tessellation when explaining his vortex theory about the universe's structure. Peter Gustav Lejeune Dirichlet used two-dimensional and three-dimensional Voronoi diagrams to study quadratic forms in 1850. In Britain,  the physician John Snow used a Voronoi diagram in 1854 to establish that most people who died in the Broad Street cholera outbreak lived closer to the infected Broad Street water pump than any other (\cite{snow1856cholera}).
Voronoi diagrams are named after Georgy Feodosievych Voronoi, who defined and studied the general n-dimensional case in 1908 (see \cite{aurenhammer2013voronoi} for more applications and details on the history of the subject).

Voronoi tessellations are defined straightforwardly using the distance function as follows. Consider a finite set of points in a metric space.
Two arbitrary points in the space are in the same cell if they have a common closest point to the original set of given points. 
More precisely, we have the following definition:

\begin{definition}
Given $N$ points (generators), $P=\{x_1,...x_N\}$ on $M^n$ an $n$-dimensional Riemannian manifold (in particular $\mathbf{R}^n$), the $i-th$ Voronoi cell is defined as
\begin{equation}
    V_{i} = \{ x \in M^{n} | d(x, x_{i}) < d(x, x_{j}) \text{ for } i \neq j \}.
\end{equation}
\end{definition}

In \cite{wormser2008generalized,boots2009spatial,aurenhammer2013voronoi} and references therein, an up-to-date description of the methods available in computational geometry can be found for computing the Voronoi cells. Computing Voronoi cells if the generators are not points, but general sets is a more intricate problem. In \cite{larsson2014iterative}, the author introduces an algorithm for computing the measure of the Voronoi cells of generators of arbitrary codimension in Euclidean space of arbitrary dimension. The algorithm relies on numerically solving a specific eikonal equation. 

This connection to the eikonal equation allows characterizing the boundary of Voronoi cells as follows: Consider $n$ point sources located
at the points, $x_i$, and suppose that a spherical front starts from each of them at time $t=0$. We consider the set of points where two fronts coming from different sources meet as time progresses (see figure \ref{wavefront}), i.e., the "collision" set of the fronts. The case of two points is easily pictured and determines the perpendicular bisector of the segment joining $x_1$ and $x_2$. This set contains the boundary of the Voronoi tessellation since, in order for two fronts to collide at a certain point, they have to be equidistant from at least two of the originally given points. From this set, we can obtain the boundary of the tessellation by removing those points in the collision set that belong to the interior of a Voronoi cell. 

\begin{figure}[ht]
\label{wavefront}
\centering
  \includegraphics[width=\linewidth]{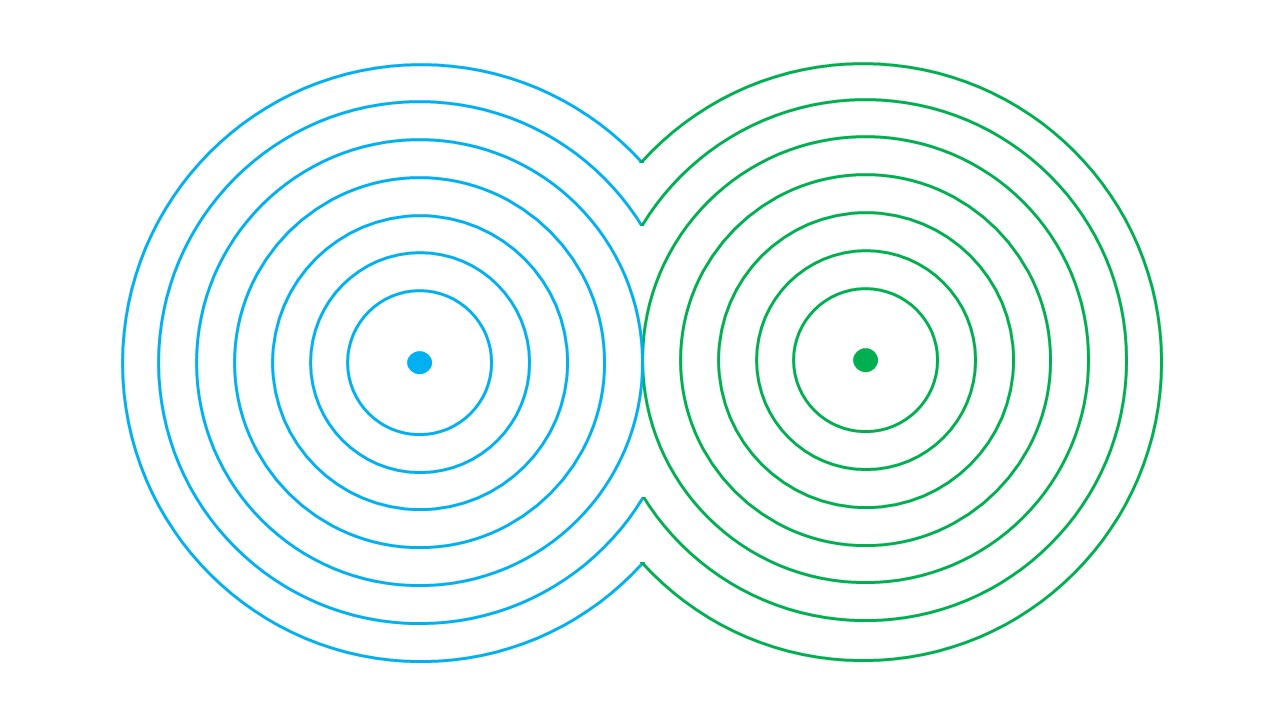}\\
  \caption{Two fronts coming from different sources (blue and green) that meet as time progresses. }\label{fronts}
\end{figure}

This principle is obvious but has several advantages. First, it can be used to characterize the boundary of the Voronoi cells as the singular set of a weak solution to a Hamilton-Jacobi equation, the eikonal equation. In \cite{mantegazza2003hamilton}, it was proved that the distance function could be characterized as the unique viscosity solution of a particular eikonal equation within a class of functions satisfying reasonable technical conditions. In \cite{mantegazza2003hamilton}, it was also shown that the set of singularities is rectifiable. Their results are valid in a broad class of Riemannian manifolds, including Euclidean space of arbitrary dimension. Second, this principle has a natural and straightforward numerical implementation. Third, it can be interpreted through diffusive waves, providing a natural interpretation of biological and physical phenomena. Recently this principle has been exploited in \cite{de2015calculating} to compute Voronoi tesselations using simple chemical reactions. The procedure in \cite{de2015calculating} has the advantage that the time needed to compute a Voronoi tessellation is independent of the number of points. This last approach, in terms of diffusion, also allows us to study the problem from the point of view of stochastic processes and to formulate a dynamic representation of Voronoi diagrams using random walks.

Voronoi cells are also the object of study in probability, in the context of percolation, see \cite{tassion2016crossing,bollobas2006critical,aizenman1998scaling,balister2005percolation,benjamini1998conformal}. The setting is this: on the plane, there is a random (Poisson) point process, and then the Voronoi tessellation corresponding to this point configuration is formed. Each cell is colored black or white with probability $p$ or $1-p$. A fundamental question is: what is the probability of an infinite black and connected set? 

A particular class of Voronoi tessellations is \emph{centroidal} Voronoi tessellations. These are Voronoi tesselations in which the generators are also the centers of mass of each cell. Such tessellations have been widely studied in the computer science literature: \cite{du1999centroidal,du2006convergence,du2002grid,du2003voronoi}. A common problem is to find points on a given set such that the resulting tessellation is centroidal. 

Voronoi tessellations can also be used to construct a mesh for a numerical solution of a PDE; this has been done both in the context of pure analysis and applications: \cite{guittet2015solving,boscheri2013semi,ghosh2011micro,gaburro2020high,mostafavi2009representing}.

Voronoi tessellations are naturally linked to another geometric object called Delaunay triangulation. Given a set of points $x_{i}$, a triangulation is a set of triangles such that the union covers the convex hull of $x_{i}$ and such that the vertices of the triangles belong to the set $x_{i}$. A triangulation is called Delaunay if the circumcircle does not contain any other point in the set $x_{i}$ for each triangle in the triangulation. The connection is this: given a Delaunay triangulation, one defines the triangles' circumcenters (intersection of the perpendicular bisectors), $y_{i}$. One then draws a line between $y_{i}$ and $y_{j}$ if and only if their corresponding triangles are adjacent. The resulting graph is the boundary of the Voronoi tessellation generated by points $x_{i}$. Given this duality, the study of Voronoi tessellations and Delaynnau triangulations are often linked (see \cite{aurenhammer2013voronoi}). 

Voronoi patterns are common in many natural systems, from biological tissues and organisms to geological formations and crystal growth \cite{Jungck2021Voronoi}, \cite{Kang2017}. These patterns emerge due to the optimization of local interactions and resource allocation, forming different cells or regions connected by a network of boundaries that minimize overlap and maximize proximity. One example of such patterns can be found in honeybee nests, where the Voronoi tessellation forms hexagonal cells to maximize geometric efficiency. Another example can be seen in the arrangement of leaves on a stem, where Voronoi patterns have been observed to optimize access to sunlight and space. Voronoi patterns are also observed in various areas of physics, from condensed matter physics to fluid dynamics \cite{Frenning2022}. These patterns are aesthetically pleasing and provide insights into the mechanisms of growth and development in natural systems.

We present in section 2 a colonization game and its relation to the solution of an elliptic equation based on the probabilistic representation just mentioned. We also consider another related problem, namely a harmonic function on a perforated domain, and use the solution to define a harmonic tessellation (as defined in this work). Section 3 discusses the formulation in terms of diffusive waves motivated by some applications to pattern formation in biological, chemical, and physical phenomena. We present some results based on short-time estimates of the heat kernel. We include a short discussion on nonlinear parabolic equations, e.g., the KPP or the porous medium equation exhibiting spherical fronts as natural candidates to be studied analytically and numerically. We devote Section 4 to characterizing Voronoi cells in terms of spherical waves using the Huygens principle. Finally, in section 5, we study a simple transportation problem that gives rise to a Voronoi tessellation and its connection to the Monge-Ampère equation.  
The last section is devoted to open questions and concluding remarks.

 \section{Elliptic equations and Voronoi cells}
 In the first part of this section, we consider the elliptic problem associated with the stationary concentration of a chemical diffusing on the plane resulting from $n$
 point sources of equal intensity. We discuss its relation with Voronoi cells motivated by a colonization game, show numerical simulations, and describe how it can model different systems, including bacterial growth and urban dynamics.
 The second part introduces a tessellation induced by the graph of a harmonic function on the perforated plane. 
 In order to fix ideas, we start by considering the stationary concentration of a substance diffusing on the plane from $n$ sources. More precisely, let $x_1,...,x_n$ be as above and consider the elliptic problem
 \begin{equation}
    \begin{split}
     \Delta u &= \sum_{i=1}^n \delta(x-x_i) \quad{\rm on}\quad \mathbb{R}^2\\
        u&\to 0 \quad{\rm as}\quad x \to\infty.
    \end{split}
\end{equation}
If $\varphi(x)$ is the fundamental solution for Laplace's equation, then the solution to the above problem 
can be written as
\begin{equation}
u(x)=\sum_{i=1}^n \varphi(x-x_i).    
\end{equation}
This solution has a natural probabilistic interpretation: it is (up to a multiplicative normalization constant) the stationary probability distribution of many Brownian walks starting at each source. Moreover, each Brownian walk is labeled with a different color, according to the source where it started. In that case, it is natural to expect that when the number of Brownian paths tends to infinity, on average, the solution will be colored as a Voronoi tessellation. This phenomenon suggests the following colonization game: From each point, $x_i$ an equal number of Brownian explorers start each one claiming the territory it moves around until it finds an explorer of a different color in its $\varepsilon$-neighborhood. Then, they stop and are replaced by two new explorers that emerge from their respective sources.

We formalize the previous discussion.
Consider $k$ Brownian motions starting at each of the $n$ points $x_i$, that is, $kn$ Brownian motions, and denote them by 
\begin{equation}
    X_{l,m}(t),
\end{equation}
with
\begin{equation}
    X_{l,m}(0)= x_l, 
\end{equation}
for each $m$, and where $l=1,...,n$ and $m=1,...,k$.
For $\varepsilon>0$ we consider walks
\begin{equation}
    X_{l,m}(t) \, \mbox{ and }\, X_{r,s}(t),
\end{equation}
for $t<T$, and we stop them if
\begin{equation}
   \mathrm{dist}(X_{l,m}(t),X_{r,s}(t'))<\varepsilon,
\end{equation}
for some $t,t' \in [0,T]$, and restart them at $x_l$ and $x_r$ respectively.

We define, for fixed $T$,$\delta$ the function
\begin{equation}
   C_\delta:\mathbb{R}^2\to \mathbb{N}^n
\end{equation}
as $(C_\delta(x))_{i}$ is the number of particles from source $x_{i}$ that have come at distance smaller than or equal to $\delta$ of the point $x$. We define $C_\delta^M: \mathbb{R}^2\to \mathbb{N}$ as 
\begin{equation}
    C_\delta^M := \mathrm{argmax}_{i} (C_\delta(x))_{i}
\end{equation}
The claim is that $C_\delta^M$ determines the Voronoi tessellation 
corresponding to the generators $x_i$, when $k\to \infty$:
\begin{equation}
    \lim_{k\to\infty} (C_\delta^M)^{-1}(i)= V(x_i),
\end{equation}
for $\delta$ sufficiently small. As before, $V(x_i)$ is the Voronoi cell associated with $x_i$.

The following algorithm is based on the description above:

\begin{itemize}
\item[a)] Domain. We define a rectangle ${\bf R}=\{(x,y):0\leq x\leq R, 0\leq y\leq U\}$.
\item[b)] Sources and particles: We randomly choose $n$ points in ${\bf R}$; these points will be the sources. $m$ particles will emerge from each source and move randomly with a fixed step size.
\item[c)] Coalitions: When a particle $P$ from source $S_i$ collides with the particle's trajectory from a different source or with the boundary, $P$ will be sent to its source $S_i$ and will resume moving from there.
\item[d)] Iterations: The particles will move synchronously and stop after a certain number of iterations.
\end{itemize}

The parameters we used are shown in table (\ref{AlgPar}).

\begin{table}
\centering
\label{AlgPar}
\begin{tabular}{|l|r|}
\hline
 \multicolumn{2}{|l|}{Parameter\quad\quad\quad\quad\quad\quad\quad\quad\quad\, Value} \\
 \hline
   Number of sources.             & $3,4,5$ \\
   Number of particles per source & $100$\\
   Number of iterations           & $300-500$\\
   Steps of the particles.        & $\alpha=1/10$ where $\alpha\in(1-,1)$\\
   Domain ${\bf R}$                       & ${\bf R}=\{(x,y):x,y\in[0,2]\}$ \\
   \hline
    \end{tabular}
    \caption{Algorithm parameters}.
\end{table}

Typically, 300 to 500 iterations and 100 particles are sufficient to recover Voronoi cells.

The following figures (\ref{voronoi4},\ref{voronoi5},\ref{voronoi6}) show the output of several simulations of our program with a varying number of sources and the number of iterations. In all of them, we draw with colored points the trajectories of the particles. The ones that emerge from a given cell have the same color. On top of the colored points, we draw Voronoi cells, computed using the Delaunay triangulation. The black circle centers correspond to the sources and are drawn so that they can easily be identified. 

\begin{figure}[ht]
\centering
  \includegraphics[width=\linewidth]{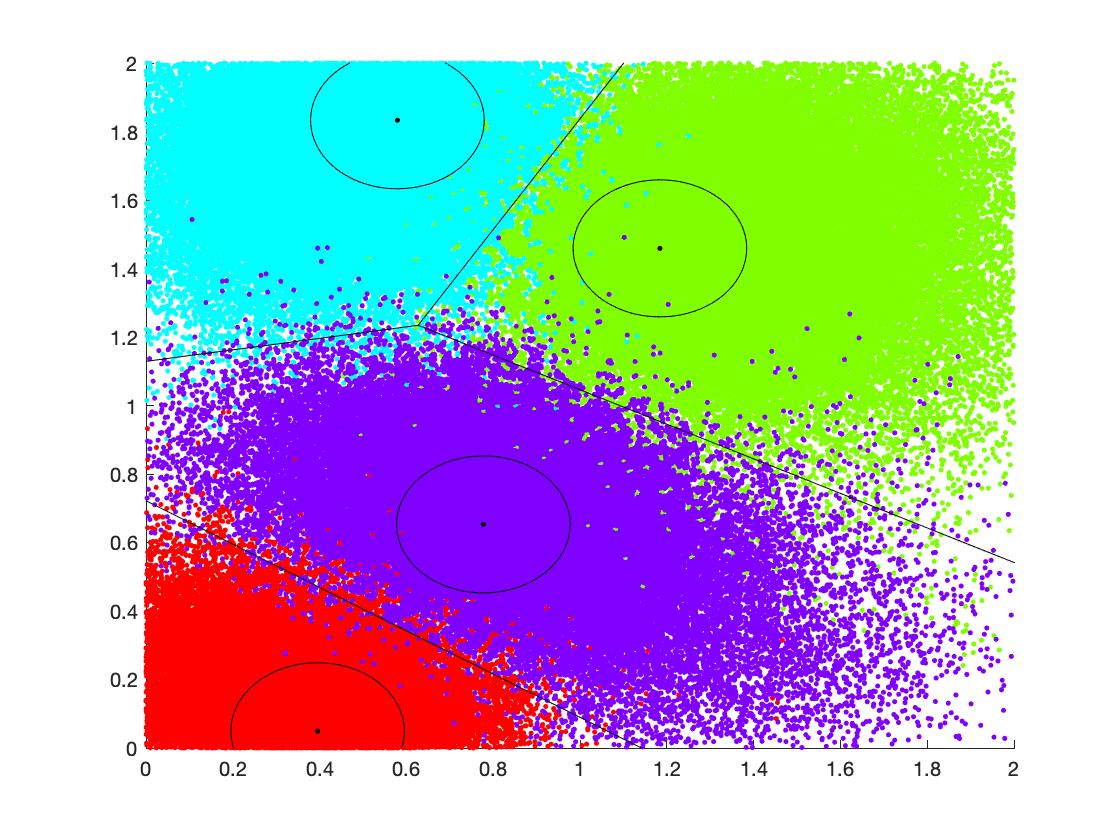}\\
  \caption{Simulation with four sources, 100 particles per source, 800 iterations.  }\label{voronoi4}
\end{figure}

\begin{figure}[ht]
\centering
  \includegraphics[width=\linewidth]{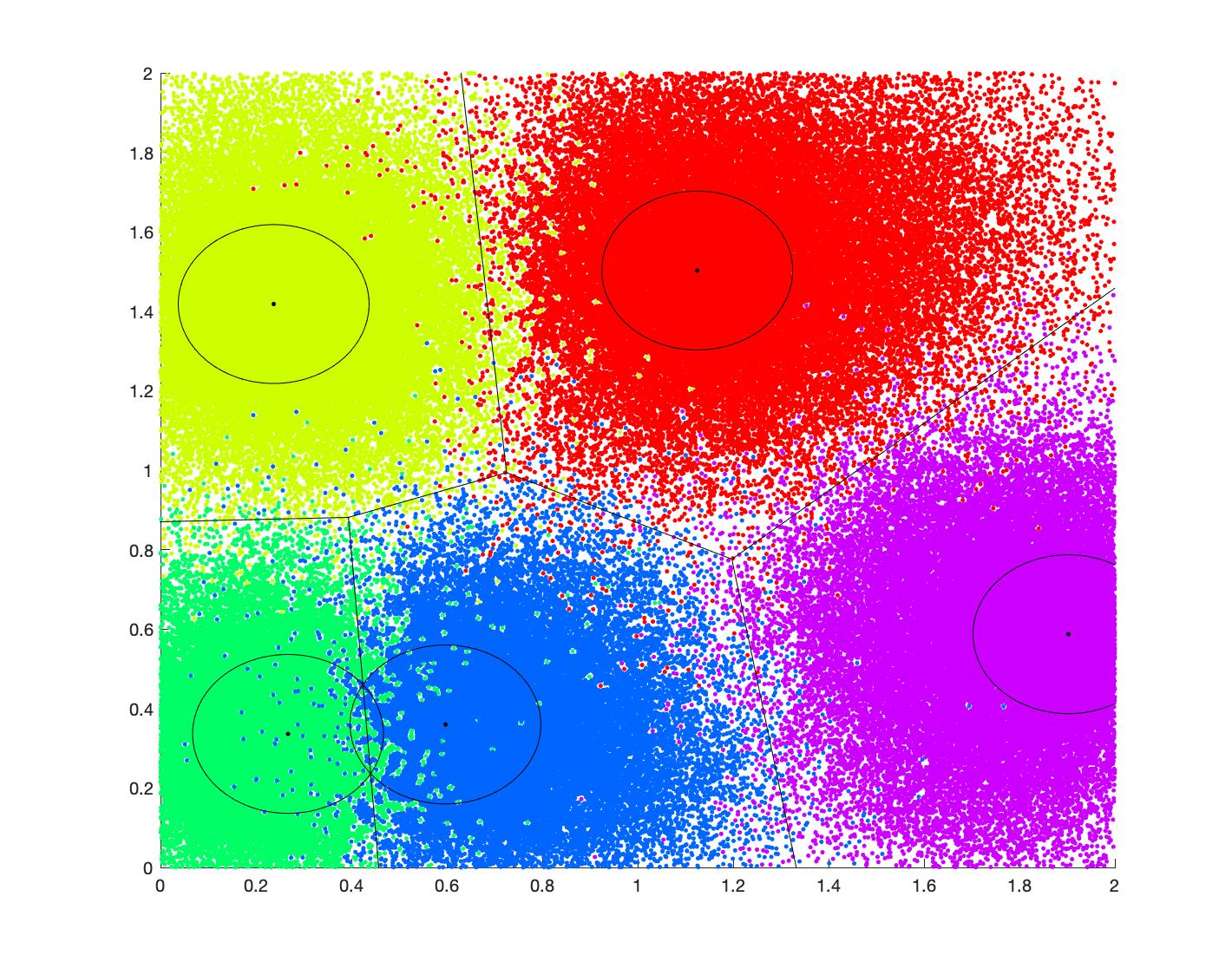}\\
  \caption{Simulation with five sources, 100 particles per source, 500 iterations.  }\label{voronoi5}
\end{figure}

\begin{figure}[ht]
\centering
  \includegraphics[width=\linewidth]{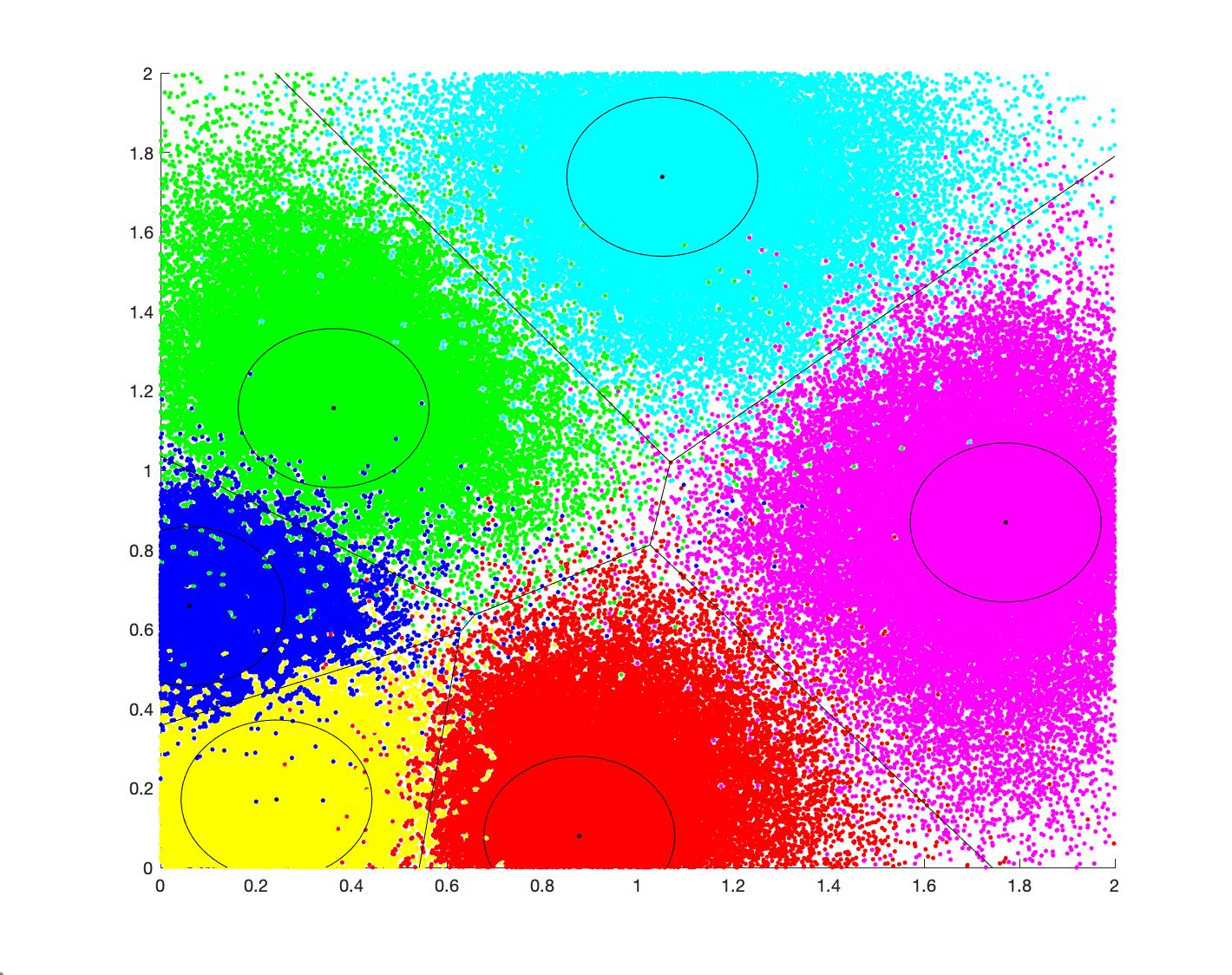}\\
  \caption{Simulation with six sources, 100 particles per source, 500 iterations.  }\label{voronoi6}
\end{figure}

All figures show that Voronoi cells are recovered from the random walkers following the abovementioned rules.
This colonization process takes place in biological and social systems and chemical reactions, as can be seen, for instance, in \cite{de2015calculating}, where a chemical reaction is documented starting with a diffusion front from $n$ points replicating the formation of Voronoi patterns on the plane. It has also been observed in colonies of bacteria starting to grow from different locations
in a Petri dish with limited resources that the end distribution of the population is a Voronoi tessellation. 
Finally, some authors have considered urban growth from different commercial centers and have found that the resulting neighborhoods correspond to  Voronoi cells. Voronoi diagrams have also been used in the context of social planning (for more details on the subject, see \cite{10.1007/978-90-481-3660-5_3}, and \cite{5362393}).
Considering models with nonlinear diffusive fronts in a more realistic scenario would be interesting, such as the KPP or the porous medium equation.

In what follows, we consider a problem motivated by the previous setting that will allow us to define a tessellation determined by a harmonic function. Before doing so, we notice that given a point $x$ on the plane,  another way of determining to what Voronoi cell it belongs consists in following the steepest ascent path determined by the distance function to the set $\bigcup_{i=1}^n \{x_i\}$. For almost all points on the plane, this path will end at the $x_i$, precisely the cell to which $x$ belongs. So let $u$ be a solution to the problem.

\begin{equation}
    \begin{split}
     \Delta u &= 0 \quad {\rm in}\quad \Omega \\
     u&=1 \quad {\rm on} \quad \partial\Omega\\
      \Omega&= \mathbb{R}^2\setminus \bigcup_{i=1}^n B_\varepsilon(x_i)\\
        u&\to 0 \quad{\rm as}\quad x \to\infty.
    \end{split}
\end{equation}

By the maximum principle, the maximum of $u$ is attained at the boundary of $\Omega$. Thus, if we follow the path of the steepest ascent starting at a given arbitrary x, it will lead to a point in $B_\varepsilon(x_i)$ for some $i$, for almost every point on the plane.
For $\varepsilon$ sufficiently small, this procedure provides the same result for a given $x$, determining a tessellation that, for simplicity, we call a harmonic tessellation.

 \section{Voronoi patterns and the heat equation}
 This section considers a mathematical model for the experiment presented in \cite{de2015calculating}. In this experiment, $n$ different diffusive sources are simultaneously placed on a Petri dish. As a result, the diffusion front can be tracked visually and stops whenever it meets a front coming from another source. Notice that it is a very similar process to the one described in the previous section.
As a first approximation, we propose a linear heat equation to model this phenomenon and prove that for sufficiently short times, the effective concentration associated with a diffusion source is restricted to the corresponding  Voronoi cell (see theorem \ref{heateqandvoronoicells}). As mentioned in the previous section, we suggest improving our model, including a nonlinear term, and consider, for instance, a Fisher-KPP type equation.
The main result in this section is the following theorem.
\begin{theorem}\label{heateqandvoronoicells}
Let $M^{n}$ be a smooth Riemannian manifold of dimension $n$, and let $\{ x_{i} \}_{i=1}^{N}$ be $N$ distinct points in $M^{n},$ let $\{ V_{i} \}_{i=1}^{N}$ be the corresponding Voronoi cells. 

Consider  $u(x,t)$ the solution of the heat equation with the initial condition
\begin{equation}
    u(x,0) = \sum_{i=1}^{N} w_{i} \delta_{x_{i}},
\end{equation}
where $w_{i} >0.$

Let $S^{n}$ be the $n-1$ unit sphere in $n$ dimensions and let $\theta \in S^{n}.$ Assume that $V_{i}$ is compact and contained in the complement of the cut locus of $x_{i}$ (see \cite{petersen2006riemannian}, section 9.2 for the definition of cut locus). Then for all $\delta, \epsilon > 0$ there exists $T>0$ such that for all $t < T,$ and for every $i,$
\begin{equation}
    u(\text{exp}_{x_{i}}( s \theta),t)
\end{equation}
is decreasing in $s$ as long as $s > \delta$ and 
\begin{equation}
    d(\text{exp}_{x_{i}}( s \theta), \partial V_{i}) > \epsilon.
\end{equation}

Furthermore, $T$ can be chosen as a continuous and monotone (in each component) function of $\delta$ and $\epsilon.$

\end{theorem}

\begin{corollary}
\label{Cor}
Let $\epsilon, \delta, T$ be as above. Let $\gamma(s)$ be a smooth path joining $x_{i}$ and $x_{j}$ for $i \neq j.$ Assume that, whenever $\gamma(s) \in V_{i},$ that 
\begin{equation}
    \frac{d}{ds} d(\gamma(s), x_{i}) \neq 0.
\end{equation}

Then any local minimums of $u(\gamma(s), t)$ occur when
\begin{equation}
    d\left(\gamma(s), \bigcup_{i=1}^{N} \partial V_{i} \right) \leq \epsilon.
\end{equation}
\end{corollary}

\begin{remark}
    We do not expect that Theorem \ref{heateqandvoronoicells} and Corollary \ref{Cor} are true for general $t$, only small $t$. This is because, for large $t$, $u$ approaches a constant function (see \cite{garcia2019concentration}). 

    Understanding more quantitatively the dependence of $T$ on $\epsilon$ and $\delta$ would require a more quantitative understanding of the function $\varphi$ in Lemma \ref{comparisonofdistances}.
\end{remark}

We need three lemmas to give the proof of the Theorem \ref{heateqandvoronoicells}. 
\begin{lemma}\label{comparisonofdistances}
Let $M^{n}$ be a smooth compact manifold, and let $\{ x_{i} \}_{i=1}^{N}$ be $N$ distinct points in $M^{n}.$ Let $\{ V_{i} \}_{i=1}^{N}$ be the corresponding Voronoi cells. 

Let 
\begin{equation}
    B_{i}^{\epsilon} = \bigcup_{x \in \partial V_{i}} B(x, \epsilon) 
\end{equation}
and let
\begin{equation}
    V_{i}^{\epsilon} = V_{i} \setminus B_{i}^{\epsilon}. 
\end{equation}

Assume that $V_{i}$ is bounded. Then there exists a strictly positive, non-increasing continuous function $\varphi(\epsilon)$ such that for any $x \in V_{i}^{\epsilon}$ and any $j 
\neq i,$
\begin{equation}
    d^{2}(x, x_{j}) \geq d^{2}(x, x_{i}) + \varphi(\epsilon).
\end{equation}
\end{lemma}

\begin{proof}
Consider the function 
\begin{equation}
    \Phi(x) = \min_{j \neq i} \{d(x,x_{j}) - d(x, x_{i})\}.
\end{equation}

Since $V_{i}$ is bounded, $V_{i}^{\epsilon}$ is compact; and therefore $\Phi$ achieves a minimum in $V_{i}^{\epsilon},$ since $\Phi$ is clearly continuous. By definition of the Voronoi cell, we have that
\begin{equation}
    \Phi(x) \geq 0,
\end{equation}
and therefore
\begin{equation}
    \min_{ y \in V_{i}^{\epsilon}} \Phi(y) \geq 0.
\end{equation}

We now claim that 
\begin{equation}
    \min_{ y \in V_{i}^{\epsilon}} \Phi(y) >0.
\end{equation}

To see this, proceed by contradiction, and assume that 
\begin{equation}
    \min_{y\in V_{i}^{\epsilon}} \Phi(y) = 0.
\end{equation}

Then there exists $x \in V_{i}^{\epsilon},$ and $j \neq i$ such that
\begin{equation}
    d(x, x_{j}) = d(x, x_{i}).
\end{equation}

Since $x_{i} \neq x_{j}$ this implies that there exists $y \in B(x, \epsilon)$ such that 
\begin{equation}
    d(y, x_{j}) < d(y, x_{i}),
\end{equation}
but this is a contradiction because $B(x, \epsilon) \in V_{i}$ and therefore 
\begin{equation}
    d(y, x_{j}) \geq d(y, x_{i}).
\end{equation}

Let 
\begin{equation}
    \alpha(\epsilon) = \min_{y \in V_{i}^{\epsilon}} \Phi(y),
\end{equation}
then for any $x \in V_{i}^{\epsilon},$
\begin{equation}
    \begin{split}
        d^{2}(x, x_{j}) &\geq (d(x, x_{i}) + \alpha(\epsilon))^{2} \\
        & = d^{2}(x, x_{i}) + 2  \alpha(\epsilon)d(x, x_{i}) + \alpha(\epsilon)^{2} \\
        & > d^{2}(x, x_{i}) + \alpha(\epsilon)^{2}.
    \end{split}
\end{equation}

Then
\begin{equation}
    \varphi(\epsilon) = \alpha^{2}(\epsilon)
\end{equation}
satisfies that
\begin{equation}
    d^{2}(x, x_{j}) \geq d^{2}(x, x_{i}) + \varphi(\epsilon).
\end{equation}

$\varphi$ is non-increasing. We will show that it is continuous by showing that it is continuous from the right and the left. We will first show continuity from the left. Let 
\begin{equation}
    \epsilon_{n} \to \epsilon
\end{equation}
monotonically from the left. Clearly 
\begin{equation}
    \varphi(\epsilon) \geq \lim_{n \to \infty} \varphi(\epsilon_{n})
\end{equation}
since 
\begin{equation}
    V_{i}^{\epsilon} \subset V_{i}^{\epsilon_{n}}
\end{equation}
and therefore
\begin{equation}
    \varphi(\epsilon) \geq \varphi(\epsilon_{n}).
\end{equation}

Now let
\begin{equation}
    x_{n} = \text{argmin}_{y \in V_{i}^{\epsilon_{n}}} \Phi(y).
\end{equation}

Modulo a subsequence, 
\begin{equation}
    x_{n} \to x \in V_{i}^{\epsilon},
\end{equation}
and by continuity of $\Phi,$
\begin{equation}
    \begin{split}
        \alpha(\epsilon) &\leq \Phi^{2}(x) \\
        &= \lim_{n \to \infty} \Phi^{2}(x_{n}) \\
        &= \lim_{n \to \infty} \alpha(\epsilon).
    \end{split}
\end{equation}

Hence,
\begin{equation}
    \varphi(\epsilon) \leq \lim_{n \to \infty} \varphi(\epsilon_{n}),
\end{equation}
therefore $\varphi$ is continuous from the left. To show that it is continuous from the right, let
\begin{equation}
    \epsilon_{n} \to \epsilon
\end{equation}
from the right. Then
\begin{equation}
    V_{i}^{\epsilon_{n}} \subset V_{i}^{\epsilon}
\end{equation}
and 
\begin{equation}
    \varphi(\epsilon) \leq \lim_{n\to \infty} \varphi(\epsilon_{n}).
\end{equation}

Let
\begin{equation}
    x = \text{argmin}_{y \in V_{i}^{\epsilon}} \Phi(y),
\end{equation}
and let $x_{n} \in V_{i}^{\epsilon_{n}}$ be such that 
\begin{equation}
   \lim_{n \to \infty} x_{n} = x.
\end{equation}

Then
\begin{equation}
    \begin{split}
        \alpha(\epsilon) &= \Phi^{2}(x) \\
        &= \lim \Phi^{2}(x_{n}) \\
        &\geq  \lim_{n \to \infty} \min_{y \in V_{i}^{\epsilon_{n}}} \Phi^{2}(y)\\
        &= \lim_{n \to \infty} \alpha(\epsilon_{n}).
    \end{split}
\end{equation}

Therefore $\varphi$ is continuous from the right and the left and hence continuous. 
\end{proof}

Before giving the proof of the next lemma, we need the definition of the fundamental solution to the heat equation:
\begin{proposition}
(\cite{garcia2019concentration}, Proposition 4.1)
Let $M^{n}$ be an $n-$dimensional smooth manifold. There exists a unique differentiable function
\begin{equation}
    p: (0,\infty) \times M^{n} \times M^{n} 
\end{equation}
denoted by $p_{t}(x,y)$ such that
\begin{equation}
    \frac{\partial}{\partial t} p_{t}(x,y) = \Delta_{y} p_{t}(x,y)
\end{equation}
and 
\begin{equation}
    \lim_{t \to 0} p_{t}(x, \cdot ) = \delta_{x}
\end{equation}
for every $x,y \in M^{n}$ and $t > 0$, where $\Delta$ is the Laplace-Beltrami operator. Such a function will be called the heat kernel for $\Delta$. It is
non-negative, it is mass preserving, i.e.
\begin{equation}
    \int_{M^{n}} p_{t}(x,y) d \pi(y) = 1
\end{equation}
for every $x \in M^{n}$ and $t > 0$, it is symmetric, i.e.
\begin{equation}
    p_{t}(x,y) = p_{t}(y,x)
\end{equation}
for every $x,y \in M^{n}$ and $t >0$ and it satisfies the semi-group property i.e.
\begin{equation}
    \int_{M^{n}} p_{t}(x,y) p_{s}(y,z) d \pi (y) = p_{t+s}(x,z)
\end{equation}
for every $x,y \in M^{n}$ and $t,s>0$. 
\end{proposition}

\begin{lemma}\label{upperboundonderivatives}
Let $M$ be a manifold, and let $p_{t}(x,y)$ be the fundamental solution to the heat equation. Then 
\begin{equation}
    \log | \nabla_{y} p_{t} (x,y) | \leq C + \left( n - 1 \right) \log(t) - \frac{d^{2}(x,y)}{4t} + \log\left( d(x,y) \right),
\end{equation}
where the constant $C$ depends only on $M^{n}.$
\end{lemma}

\begin{proof}
Let 
\begin{equation}\label{definitionofL}
    L_{x}^{t} (y) = t \log (p_{t} (x,y)).
\end{equation}

Then by chain rule
\begin{equation}
\begin{split}
    \nabla_{y} p_{t} &= \frac{1}{t} \nabla  L_{x}^{t} \exp\left( \frac{1}{t}   L_{x}^{t}  \right) \\
    &=  \frac{1}{t} p_{t} (x,y) \nabla  L_{x}^{t}. 
\end{split}    
\end{equation}

Using lemma 4.7 of \cite{garcia2019concentration} we have that, for $t \in (0,1),$
\begin{equation}
    p_{t}(x,y) \leq \frac{C}{ t^{n -\frac{1}{2}}} \exp \left( \frac{d^{2}(x,y)}{4t} \right),
\end{equation}
where $C$ depends only on $M^{n}.$

On the other hand, theorem 5.5.3 of \cite{hsu2002stochastic} implies that 
\begin{equation}
  \frac{1}{t}  | \nabla L_{x}^{t}| \leq C \left[ \frac{d(x,y)}{t} + \frac{1}{\sqrt{t}} \right],
\end{equation}
where $C$ depends only on $M^{n}.$ Hence,
\begin{equation}
    | \nabla L_{x}^{t}| \leq C \left[ d(x,y) + \sqrt{t} \right].
\end{equation}

Putting everything together, we have
\begin{equation}
    |\nabla p_{t}| \leq  C \left[ d(x,y) + \sqrt{t} \right] \frac{1}{ t^{n -\frac{1}{2}}} \exp \left( \frac{d^{2}(x,y)}{4t} \right),
\end{equation}
which implies
\begin{equation}
\begin{split}
    \log | \nabla_{y} p_{t} (x,y) | &\leq C + \left( n - \frac{1}{2} \right) \log(t) - \frac{d^{2}(x,y)}{4t} - \log(t) + \log\left( d(x,y) \right) + \frac{1}{2} \log(t)\\
    &\leq C + \left( n - 1 \right) \log(t) - \frac{d^{2}(x,y)}{4t} + \log\left( d(x,y) \right).
\end{split}    
\end{equation}

\end{proof}

\begin{lemma}\label{lowerboundonderivatives}
Let $M^{n}$ be a manifold of dimension $n$, and let $p_{t}(x,y)$ be the fundamental solution to the heat equation. Let $x \in M^{n}$ and let $K \subset M^{n}$ be a compact subset of $M^{n}$ containing $x,$ and contained in the complement of the cut-locus of $x.$ Let $S^{n}$ be the $n-1$ dimensional unit sphere in $n$ dimensions, and let $\theta \in S^{n}.$ Then 
\begin{equation}
\begin{split}
   &\log \left( \frac{d}{ds} \left( p_{t} (x, \exp_{x}(s \theta)) \right) \right) \geq \\
   & C - \left( 1+ \frac{n}{2} \right)\log(t) - \frac{s}{4t} + \log(s) + \log(err),
\end{split}   
\end{equation}
where $\log(err) \to 0$ uniformly in $K$ and $C$ depends only on $M^{n}.$
\end{lemma}

\begin{proof}
    We start by quoting Corollary 2.29 from \cite{malliavin1996short}:
    
    Let $x,y \in M^{n}$ let 
    \begin{equation}
    E (x,y) = \frac{1}{2} dist^{2}(x,y).
    \end{equation}
 Then for fixed $y,$ the function $E(x, y)$ is smooth in $x$ for $ x \in M^{n} \setminus C_{M^{n}}(y),$ where $C_{M^{n}}(y)$ is the cut locus of $y.$ Moreover,
    \begin{equation}
        \lim_{ t \to 0 } [ \nabla L_{x}^{t}] = - \nabla_{y} E
    \end{equation}
uniformly on compact subsets of $M^{n} \setminus  C_{M^{n}}(y),$ where $L_{x}^{t}$ is given by equation \ref{definitionofL}.

Proceeding as in the proof of lemma \ref{upperboundonderivatives}, we have that
\begin{equation}
    \nabla_{y} p_{t} =  \frac{1}{t} p_{t} (x,y) \nabla  L_{x}^{t}. 
\end{equation}

Using lemma 4.7 of \cite{garcia2019concentration}, we have that
\begin{equation}
    p_{t} (x,y) \geq \frac{C}{t^{\frac{n}{2}}} \exp \left( \frac{d^{2}(x,y)}{4t} \right),
\end{equation}
where $C$ depends only on $M^{n}.$

Note that 
\begin{equation}
    \langle \nabla_{y} E(x,y), \frac{d}{ds} \exp_{x}(s \theta) \rangle =s 
\end{equation}
if 
\begin{equation}
    \exp_{x}(s \theta) =y.
\end{equation}

Therefore
\begin{equation}
    \begin{split}
        \frac{d}{ds}  \left( p_{t} (x, \exp_{x}(s \theta)) \right)&=  \langle \nabla_{y}p_{t}(x, \exp_{x}(s \theta)), \frac{d}{ds} \exp_{x}(s \theta) \rangle \\
        &=  \frac{1}{t} p_{t} (x,\exp_{x}(s \theta)) \langle \nabla  L_{x}^{t}, \frac{d}{ds} \exp_{x}(s \theta)) \rangle \\
         &= \frac{1}{t}  p_{t} (x,\exp_{x}(s \theta)) \cdot s \cdot err \\
         & \geq \frac{1}{t} \frac{C}{t^{\frac{n}{2}}} \exp \left( \frac{d^{2}(x,\exp_{x}(s \theta)))}{4t} \right) \cdot s \cdot err \\
         & = \frac{C}{t^{1+\frac{n}{2}}} \exp \left( \frac{s^{2}}{4t} \right) \cdot s \cdot err,
    \end{split}
\end{equation}
where $err \to 1$ uniformly in $K$ in virtue of Corollary 2.29 of \cite{malliavin1996short}.

The previous equation implies that
\begin{equation}
\begin{split}
   &\log \left( \frac{d}{ds} \left( p_{t} (x, \exp_{x}(s \theta)) \right) \right) \geq \\
   & C - \left( 1+ \frac{n}{2} \right) \log(t) - \frac{s}{4t} + \log(s) + \log(err) .
\end{split}   
\end{equation}
\end{proof}

We are ready to give the proof of theorem \ref{heateqandvoronoicells}.

\begin{proof}
    (Of theorem \ref{heateqandvoronoicells}).
    
    Note that $u_{t}$ can be written as 
    \begin{equation}
        u_{t}(y) = \sum_{j}^{N} w_{j} p_{t}(x_{j}, y).
    \end{equation}
    Without loss of generality, assume $i=1.$ Then we have that
    \begin{equation}
        \nabla u_{t}(y) = w_{1}\nabla_{y} p_{t}(x_{1}, y) + \sum_{j=2}^{N} w_{j}\nabla_{y} p_{t}(x_{j}, y).
    \end{equation}
    Let 
    \begin{equation}
        M = \max d(x_{j}, x_{k}).
    \end{equation}
    
    Using lemma \ref{upperboundonderivatives}, we have that 
    \begin{equation}
        \begin{split}
            &\left| \sum_{j=2}^{N} w_{j} \nabla_{y} p_{t}(x_{j}, y) \right| \leq \\
            & N \max_{j} \left| w_{j} \nabla_{y} p_{t}(x_{j}, y) \right| \leq \\
            &  \exp \Bigg(  C + \left( n - 1 \right) \log(t) - \frac{d^{2}(x,y)}{4t} +\\
            & \ \ \ \log\left( M\right) + \log(N) + \max_{j}\log(w_{j}) \Bigg).
        \end{split}
    \end{equation}
    
    Using lemma \ref{comparisonofdistances}, we have that
    \begin{equation}
        \begin{split}
            &\frac{d}{ds} \left(  \sum_{j=2}^{N} w_{j} p_{t}(x_{j}, \exp_{x_{1}}(s \theta)) \right) \leq \\
            &\left| \sum_{j=2}^{N} w_{j} \nabla_{y} p_{t}(x_{j}, \exp_{x_{1}}(s \theta)) \right| \leq \\
            &\exp \Bigg(  C + \left( n - 1 \right) \log(t) - \frac{s^{2} + \varphi(\epsilon)}{4t} \\
            &\ \ \ \ + \log\left( M\right) + \log(N)  + \max_{j}\log(w_{j}) \Bigg).
        \end{split}
    \end{equation}
    
    Note that for any $\epsilon, \delta >0$ and $s>\delta$ there exists $T>0$ such that, if $t <T$ then 
    \begin{equation}
        \begin{split}
             &C + \left( n - \frac{1}{2} \right) \log(t) - \frac{s^{2} + \varphi(\epsilon)}{4t} - \log(t) + \log\left( M\right) + \frac{1}{2} \log(t) + \log(N)  + \max_{i}\log(w_{i}) \leq \\
             &  C - \left( 1+ \frac{n}{2} \right)\log(t) - \frac{\delta^{2}}{4t} + \log(\delta)  + \min_{j}\log(w_{j})+ \log(err).
        \end{split}
    \end{equation}
    
    Since $\varphi$ is continuous and monotone, $T$ can be chosen as continuous and monotone in each variable. 
        
    Also, by Lemma \ref{lowerboundonderivatives}, we have that 
    \begin{equation}
        \begin{split}
           &  C - \left( 1+ \frac{n}{2} \right)\log(t) - \frac{\delta^{2}}{4t} + \log(\delta)  + \min_{j}\log(w_{j})+ \log(err) \leq  \\
           &  C - \left( 1+ \frac{n}{2} \right)\log(t) - \frac{s^{2}}{4t} + \log(s)  + \min_{j}\log(w_{j})+ \log(err) \leq \\
           &\log \left( \frac{d}{ds} \left( w_{1}p_{t} (x, \exp_{x}(s \theta)) \right) \right)
        \end{split}
    \end{equation}

    Therefore, for $\epsilon, \delta, t, T$ as above, we have that 
    \begin{equation}
      \left|  \frac{d}{ds} \left(  \sum_{j=2}^{N} w_{j}  p_{t}(x_{j}, \exp_{x_{1}}(s \theta)) \right) \right| \leq  \frac{d}{ds}  w_{1}  p_{t}(x_{1}, \exp_{x_{1}}(s \theta)),
    \end{equation}
    
    and therefore
    \begin{equation}
        \frac{d}{ds} \left( u_{t} (\exp_{x_{1}}(s \theta)) \right) < 0. 
    \end{equation}
\end{proof}

\begin{remark}
The dependence of $T$ on the parameters is quite delicate. The choice of $T$ will depend on the number of points, the maximum distance between them, and the weights. However, $T$ also depends on the function $\varphi,$ which is closely linked to the geometry of the Voronoi cells. 
\end{remark}

\section{Huygens principle and Voronoi cells in optics}
For the sake of completeness, we also include some comments about hyperbolic equations, although the following remarks
are straightforward.

The same dynamic construction mentioned in the previous sections relative to diffusive spherical fronts applies to the propagation of light in optics from $n$ given sources utilizing the Huygens principle;
this can be formalized using the eikonal equation:

\begin{equation}
    |\nabla u|^2=1,
\end{equation}
on the plane. It is well known that the distance function $d$ constitutes a solution. Given $n$ points on the plane,
the distance function to this set, $P=\{x_1,...x_N\}$
\begin{equation}
    d(x,P),
\end{equation}
is a singular solution, which is smooth in the interior of a Voronoi cell and whose singular set is precisely the boundary of the Voronoi cells. The same observation holds in higher dimensions and Riemannian manifolds, and we could use the results in \cite{mantegazza2003hamilton} to conclude that the singular set is rectifiable.

\section{Voronoi patterns, optimal transport, and the Monge-Ampere equation} 

Voronoi tessellations appear naturally in applications to urban planning. For example, consider the following simple problem: each household must be assigned exactly one school in a city with uniform population density. The optimal configuration of schools is the one that implies the least distance traveled by households. The solution is given by partitioning the city into the Voronoi tessellation generated by the schools. One may consider a similar problem in which, additionally, each school has a capacity limit. In this case, the solution is given by partitioning the city into the \emph{power set} (see equation \eqref{eq:powerset}) generated by the schools, with weights specified by their limit capacity, see \cite{galichon2018optimal} for an in-depth discussion. See also \cite{10.1007/978-90-481-3660-5_3} for a discussion of the implementation of this strategy and an analysis of a case study in Shandong Province, China. 

This problem can be formulated in terms of optimal transportation: Given a domain $\Omega$, point $x_{i} \in \Omega$, and positive real numbers $\alpha_{i}$ such that $\sum_{i} \alpha_{i} = 1$, the problem is to optimally (with quadratic cost) transport the probability measure $\sum \alpha_{i}\mathbf{1}_{x_{i}}$ to the uniform probability measure on $\Omega$. In \cite{aurenhammer1998minkowski}, it was shown that this problem is equivalent to the problem of finding a power set generated by the points $x_{i}$. In particular, given specified volume fractions, weights always exist such that the cells in the resulting power diagram have specified volume fractions. \cite{aurenhammer1998minkowski} Also introduced a new algorithm to compute power diagrams. 

We construct a family of solutions to the Monge-Ampere equation based on these results. This family of solutions is a counterexample to several results for the Monge Ampere equation (particularly concerning regularity) when the hypotheses are slightly modified. 

\begin{theorem}\label{mongeampereeeqtheorem}
Let $\Omega \subset \mathbf{R}^{n}$ be an open bounded set and let $\{w_{i}\}_{i=1}^{N} \in \mathbf{R}.$ Let $\{ x_{i} \}_{i=1}^{N}$ be $N$ distinct points in $\Omega.$ Fix $\lambda \in (0,1)$ and let 
\begin{equation}
    \Phi_{\lambda} (x) = \frac{1}{2} | x |^{2} + \frac{\lambda - 1}{2} \min_{i} \{ | x - x_{i}|^{2} + w_{i} \}.
\end{equation}

Then $\Phi_{\lambda}$ is a strictly convex function that solves 
\begin{equation}\label{mogeampereequation}
    \det (\nabla^{2} \Phi_{\lambda} ) = \lambda^{n}
\end{equation}
in the Brennier sense (Definition \ref{Brennier}).
\end{theorem}

\begin{definition}\label{Brennier}
A convex  function $\varphi$ is a solution of 
\begin{equation}\label{mongeampereeq2}
    \det(\nabla^{2} \varphi) = \frac{f(x)}{g(\nabla \varphi (x))}
\end{equation}

in the Brennier sense if 

\begin{equation}\label{distributionalmongeampereeq}
    \int \xi(y) g(y) \, dy = \int \xi(\nabla \varphi(x)) f(x) \, dx 
\end{equation}

for all $\xi \in C^{\infty}$ and $g(\nabla \varphi(x)) \neq 0.$
\end{definition}

\begin{lemma}
If $\varphi \in C^{\infty}$ is convex and satisfies \ref{distributionalmongeampereeq}  for all $\xi \in C^{\infty}$, then $\varphi$ satisfies \ref{mongeampereeq2}.
\end{lemma}

\begin{proof}
    The proof is essentially found in chapter $4$ of \cite{villani2003topics}, but we include it here for completeness. 
    
    If $\varphi \in C^{\infty}$ and is convex, we can perform a change of variables 
    \begin{equation}
        y = \nabla \varphi(x)
    \end{equation}
    on the left-hand-side and get 
    \begin{equation}
    \int \xi(\nabla \varphi(x)) g(\nabla \varphi(x)) \det(\nabla^{2} \varphi(x)) \, dx = \int \xi(\nabla \varphi(x)) f(x) \, dx 
    \end{equation}
 for all $\xi \in C^{\infty}.$ Since this is true  for all $\xi \in C^{\infty},$ we have that
     \begin{equation}
     g(\nabla \varphi(x)) \det(\nabla^{2} \varphi(x)) = f(x).
    \end{equation}
    
    Since $ g(\nabla \varphi(x)) \neq 0,$ we have that 
    
    \begin{equation}
      \det(\nabla^{2} \varphi(x)) = \frac{f(x)}{g(\nabla \varphi(x))}.
    \end{equation}
\end{proof}

We will now prove Theorem \ref{mongeampereeeqtheorem}

\begin{proof}
    (Of theorem \ref{mongeampereeeqtheorem})
    
    For simplicity, we assume that $\Omega$ has volume $1,$ (otherwise, we could apply this argument to a rescaled solution). 
    
    Let $V_{i}$ be the power set of $x_{i},$ defined as 
    \begin{equation}
    \label{eq:powerset}
        V_{i} = \{ x \in \Omega | | x-x_{i}|^{2}+w_{i} < | x-x_{j}|^{2}+w_{j} \forall j \neq i \}.
    \end{equation}
    Assume without loss of generality that $V_{i}$ is nonempty for each $i.$
  
    Consider the sets
    \begin{equation}
        V_{i}^{\lambda} = \{ x_{i} + \lambda(x-x_{i}) | x \in V_{i} \}.
    \end{equation}
    
    Consider $\mu$ the uniform probability measure on $\Omega,$ and also the probability measure $\nu,$ defined as
    \begin{equation}
        d\nu = \frac{1}{\lambda^{n}} \sum_{i=1}^{N} \mathbf{1}_{V_{i}^{\lambda}}.
    \end{equation}
    
    We can easily verify that for each $x \in V_{i},$
    \begin{equation}\label{eqforpushforward}
        \nabla \Phi_{\lambda} = x +  (\lambda -1)(x - x_{i}).
    \end{equation} 
    and therefore $\nu$ is the push-forward of $\mu$ under $\nabla \Phi_{\lambda}:$
    \begin{equation}
        \nu = \nabla \Phi_{\lambda} \# \nu.
    \end{equation}
    
    Therefore $\Phi_{\lambda}$ solves \ref{distributionalmongeampereeq} with 
    \begin{equation}
        \begin{split}
            f(x ) &=1 \\
            g(x) &=  \frac{1}{\lambda^{n}} \sum_{i=1}^{N} \mathbf{1}_{V_{i}^{\lambda}}.
        \end{split}
    \end{equation}
    
    On the other hand, we also have that 
    \begin{equation}
        g( \nabla \Phi_{\lambda} (x) ) = \frac{1}{\lambda^{n}}
    \end{equation}
    for all $x.$ Therefore $\Phi_{\lambda}$ solves \eqref{mogeampereequation} in the sense of Brennier. 
\end{proof}

\begin{remark}
One way to get intuition on this result is to first solve for $\Phi_{\lambda}$ such that 
    \begin{equation}
        \nu = \nabla \Phi_{\lambda} \# \nu.
    \end{equation}
    This leads to finding $\Phi_{\lambda}$ such that, for each $x \in V_{i},$
    \begin{equation}
        \nabla \Phi_{\lambda} = x +  (\lambda -1)(x - x_{i}).
    \end{equation}
    This leads to the formula for $\Phi_{\lambda}.$
\end{remark}

As a corollary, we recover the solution to the corresponding problem in Optimal transportation.

\begin{corollary}\label{solutiontoopttransprob}
Consider the Monge transportation problem
\begin{equation}
    \min_{T \# \mu = \nu} \int_{\Omega} \left| x - T(x) \right|^{2} d \mu,
\end{equation}
with $\mu$ the uniform measure, and $\nu$ given by
\begin{equation}
    \nu = \sum_{i=1}^{N} \frac{1}{\lambda^{n}} \mathbf{1}_{V_{i}^{\lambda}}.
\end{equation}

Then by Brennier's theorem, the solutions exist and are given by 
\begin{equation}
    T = \nabla \Phi_{\lambda}.
\end{equation}
\end{corollary}

\begin{remark}
Letting $\lambda$ tend to $0,$ we obtain that Corollary \ref{solutiontoopttransprob} remains true when 
\begin{equation}
    \nu = \sum_{i=1}^{N} |V_{i}| \delta_{x_{i}}.
\end{equation}
\end{remark}
This last remark recovers a classical result in Optimal Transportation \cite{aurenhammer1998minkowski}.    

$\Phi_{\lambda}$ is a counterexample to several results for Alexandrov solutions of the Monge Ampere equation (see \cite{mooney2015partial} for the definition of Alexandrov solution), mainly concerning regularity. We list the results to which $\Phi_{\lambda}$ is a counterexample and omit the proof.    

\begin{theorem}
(\cite{mooney2015partial} Theorem 1.2) Let $u$ be an Alexandrov solution to 
\begin{equation}
    0< \lambda \leq Det(D^{2}u ) \leq \Lambda<\infty
\end{equation}
in $B^{1} \subset \mathbf{R}^{n}.$ Then $u \in W^{2,1}_{loc}(B_{1})$
\end{theorem}

\begin{theorem}
(\cite{mooney2015partial} Theorem 1.3) Assume that $u$ and $v$ are Alexandrov solutions to 
\begin{equation}
    det(D^{2}u) = det(D^{2}v) = f
\end{equation}
in an open connected set $\Omega \subset \mathbf{R}^{n},$ with $f \in C^{1, \alpha}(\overline{\Omega})$ strictly positive. If $u=v$ in an open subset of $\Omega$ then $u=v$ in $\Omega.$ 
\end{theorem}
    
\begin{theorem}
(\cite{caffarelli1991some}) A strictly convex (Alexandrov) solution $u$ of
\begin{equation}
    det(D^{2}u) = d \, \mu,
\end{equation}
with $d \mu$ satisfying $0 < \lambda \leq d \mu \leq \Lambda < \infty$ is $C^{1,\alpha},$ with the $C^{1,\alpha}$ norm depending only on the Lipschitz norm of $u$ and on its strict convexity. \footnote{The exact statement of Theorem 2 is more general, see \cite{caffarelli1991some}}.
\end{theorem}    

\begin{theorem}
(\cite{schmidt2013w2} Theorem 1.1) Suppose that $u: \Omega_{1} \to \mathbf{R}$ is a strictly convex Alexandrov solution of 
\begin{equation}
    \det (\nabla^{2} u ) = f
\end{equation}
where $f$ satisfies 
\begin{equation}
    0<\lambda \leq f \leq \Lambda < \infty.
\end{equation}
Then we have $u \in W^{2, 1 + \epsilon}_{loc}(\Omega)$ for some positive constant $\epsilon$ which depends only on $n, \lambda, \Lambda.$
\end{theorem}

\begin{corollary}
(\cite{figalli2017monge} Corollary 2.11)
Let $\Omega$ be an open bounded set, $g:\partial \Omega \to \mathbf{R}$ a continuous function, and $\nu$ a Borel measure in $\Omega.$ Then there exists at most one convex function $u: \Omega \to \mathbf{R}$ solving (in the Alexandriv sense) the problem
\begin{equation}
    \begin{split}
        &\det(\nabla^{2}u)=\nu \text{ in } \Omega \\
        &u = g  \text{ in } \partial \Omega
    \end{split}
\end{equation}
\end{corollary}

\section{Conclusion}

In this paper, we have established some new connections between partial differential equations and Voronoi tessellations. In particular, for an elliptic problem with point sources, we presented a stochastic game whose solution is provided by a Voronoi tessellation. This problem motivated the introduction of a tessellation associated with a harmonic function. A result for the decay properties of the solution and short-time asymptotics linked to Voronoi cells has been established for parabolic equations. In the case of hyperbolic equations, the rectifiability of the boundary of the Voronoi cells has been obtained due to characterizing this boundary as the singular set of solutions to an eikonal equation. Finally, some properties of the solution to an optimal transport problem have been obtained by connecting this problem with Voronoi tessellation. Future work will focus on general Riemannian manifolds and other nonlinear equations, such as Fisher-KPP and the medium porous equations. Another possible new direction is to extend the numerical simulations to a broader class of colonization games. 

\noindent 
Extensions to objects different from points, higher dimensional Voronoi cells and Voronoi tessellations on manifolds provide interesting open questions.

\section*{Declarations}

\subsection*{\bf{Ethics approval and consent to participate}}
Not applicable

\subsection*{\bf{Consent for publication}}
Not applicable

\subsection*{\bf{Availability of data and materials}}
Not applicable

\subsection*{\bf{Competing interests}}
The authors have no competing interests.

\subsection*{\bf{Funding}}
The work done for this manuscript by PPL and YCP  was partly funded by CONACyT (CF2019-217367). In addition, DPG acknowledges support from the German Research Foundation (DFG) via the research unit FOR 3013, "Vector- and tensor-valued surface PDEs" (grant no. NE2138/3-1). 

\subsection*{\bf{Authors' contributions}}
All authors contributed equally to this project.

\subsection*{\bf{Acknowledgements}}
Not applicable

\bibliographystyle{plain}
\bibliography{bibliography.bib}

\end{document}